\documentclass[11pt]{article}  
\usepackage{amsthm}
\usepackage{amsmath,amssymb}
\usepackage{array}
\usepackage{float}
\usepackage[margin=1in]{geometry}
\usepackage{graphicx}

\usepackage{algorithm}
\usepackage[noend]{algpseudocode}

\begin{document}

\newcommand{\bE}{\ensuremath{\mathbf{E}}}
\newtheorem{theorem}{Theorem}[section]
\newtheorem{proposition}[theorem]{Proposition}
\newtheorem{lemma}[theorem]{Lemma}
\newtheorem{corollary}[theorem]{Corollary}
\newtheorem{Definition}[theorem]{Definition}

\title{Comparison of two convergence criteria for the variable-assignment Lopsided Lov\'{a}sz Local Lemma}
\author{David G. Harris\thanks{Department of Computer Science, University of Maryland, 
College Park, MD 20742. 
Research supported in part by NSF Awards CNS 1010789 and CCF 1422569.
Email: \texttt{davidgharris29@gmail.com}}}

\date{}

\maketitle

\begin{abstract}
The Lopsided Lov\'{a}sz Local Lemma (LLLL) is a cornerstone probabilistic tool for showing that it is possible to avoid a collection of ``bad'' events as long as their probabilities and interdependencies are sufficiently small. The strongest possible criterion  in these terms is due to Shearer (1985), although it is technically difficult to apply to constructions in combinatorics.

The original formulation of the LLLL was non-constructive; a seminal algorithm of Moser \& Tardos (2010) gave an efficient  algorithm for nearly all its applications, including to $k$-SAT instances where each variable appears in a bounded number of clauses. Harris (2015) later gave an alternate criterion for this algorithm to converge; unlike the LLL criterion or its variants, this criterion depends in a fundamental way on the decomposition of bad-events into variables.

In this note, we show that the criterion given by Harris can be stronger in some cases even than Shearer's criterion. We construct $k$-SAT formulas with bounded variable occurrence, and show that the criterion of Harris is satisfied while the criterion of Shearer is violated. In fact, there is an exponentially growing gap between the bounds provable from any form of the LLLL and from the bound shown by Harris.
\end{abstract}

\maketitle

\section{Introduction}
The Lov\'{a}sz Local Lemma (LLL) is a general probabilistic principle for showing that, in a probability space $\Omega$ with a finite set $\mathcal B$ of ``bad'' events which are not too interdependent and are not too likely, then there is a positive probability no events in $\mathcal B$ occur. Since its introduction in \cite{lll-orig}, it has become a cornerstone of the probabilistic method of combinatorics.

There have been numerous extensions of the LLL since its original formulation. One important generalization known as the \emph{Lopsided Lov\'{a}sz Local Lemma} (LLLL) \cite{erdos-spencer} observes that it is not necessary for bad-events to be fully independent. If the bad-events are \emph{positively correlated} in a certain sense, then for the purposes of the LLL this is just as good as independence. This type of correlation, which we discuss shortly, is known as \emph{lopsidependency}.

In order to explain the LLL formally, we need to introduce a number of definitions.

For any collection of events $S \subseteq \mathcal B$, we define  $\overline{S} = \bigcap_{B \in S} \overline B$; we refer to this event as \emph{avoiding $S$.}  A \emph{dependency graph} is a graph $G$ on vertex set $\mathcal B$ such that for any $B \in \mathcal B$ and any set $S \subseteq \mathcal B - \{B \} - N_G(B)$ (where $N_G(B)$ denotes the neighborhood of $B$ in $G$), we have
\begin{equation}
\label{e1}
\Pr( B \mid \overline S) = \Pr( B).
\end{equation}
That is, each bad-event $B \in \mathcal B$ is independent of all other events in $\mathcal B$, except possibly those which are neighbors of $B$ in the dependency graph.   A \emph{lopsidependency graph} is a graph $G$ on vertex set $\mathcal B$, satisfying the relaxed condition that for any $B \in \mathcal B$ and set $S \subseteq \mathcal B - \{B \} - N_G(B)$,
\begin{equation}
\label{e2}
\Pr( B \mid \overline S) \leq \Pr( B).
\end{equation}

A probability space $\Omega$ and collection of bad-events $\mathcal B$ does not have a unique dependency graph or lopsidependency graph. Rather, we suppose that we are given $\Omega, \mathcal B$ and some chosen graph $G$ which is a (lopsi-)dependency graph for them. 

For such a graph $G$ with vertex set $V = \mathcal B$, we say a set $S \subseteq V$ is \emph{stable} if no elements of $S$ are adjacent in $G$. For real numbers $p_v$, indexed by the vertices $v \in V$, we define the \emph{stable set polynomial} of $G$ with respect to base set $S \subseteq V$, denoted $Q(G,S, \vec p)$, by
$$
Q(G,S, \vec p) = \sum_{\substack{\text{stable sets $T$} \\ S \subseteq T \subseteq V}} (-1)^{|T| - |S|} \prod_{v \in T} p_v
$$

With these definitions, we state a few formulations of the LLLL.
\begin{theorem}
\label{lll-basic-thm}
Suppose $G$ is a lopsidependency graph for $\Omega, B$. If any of the following conditions are satisfied, then  $\Pr(\overline{\mathcal B}) > 0$.
\begin{enumerate}
\item (Symmetric LLLL) If $G$ has maximum degree $d$ and every $B \in \mathcal B$ has $\Pr(B) \leq p$ and 
$$
e p (d+1) \leq 1
$$

\item (Asymmetric LLLL) If there is a function $x: \mathcal B \rightarrow (0,1)$  satisfying
    $$
    \forall B \in \mathcal B \qquad \Pr(B) \leq x(B) \prod_{A \in N_G(B)} (1 - x(A))
    $$
    
    \item (Cluster-expansion criterion \cite{bissacot}) If there is a function $\mu: \mathcal B \rightarrow [0,\infty)$ satisfying
    $$
    \forall B \in \mathcal B \qquad  \mu(B) \geq \Pr(B) \Bigl( \mu(B) + \sum_{\substack{Y \subseteq N_G(B) \\ \text{$Y$ stable}}} \prod_{A \in Y}  \mu(A) \Bigr)
      $$
      \end{enumerate}
      \end{theorem}
 
The symmetric LLLL uses only a few crude parameters of the problem instance --- namely, the maximum probability of a bad-event and the maximum degree of the lopsidependency graph. The other variants use progressively more information and take advantage of refined dependency structure.  See also \cite{kolipaka2} for another criterion in this vein. In \cite{shearer}, Shearer derived the most powerful possible criterion in these terms.

\begin{theorem}[Shearer's criterion \cite{shearer}]
\label{shearer}
Let $G$ be a graph on vertex set $V = \{1, \dots, n \}$ and let $p_1, \dots, p_n \in [0,1]$.
\begin{enumerate}
\item Suppose that $Q(G, \emptyset, \vec p) > 0$ and $Q(G,S, \vec p) \geq 0$ for all $S \subseteq V$. Then for any probability space $\Omega$, and any events $B_1, \dots, B_n \subseteq \Omega$ in that space such that $\Pr(B_i) = p_i$ for $i = 1, \dots, n$ and such that $G$ is a lopsidependency graph for $\mathcal B = \{B_1, \dots, B_n \}$, we have $\Pr(\overline{\mathcal B}) \geq Q(G, \emptyset, \vec p)> 0$.

In this case, we say that \emph{Shearer's criterion is satisfied} by $G, \vec p$.

\item Suppose that either $Q(G, \emptyset, \vec p) \leq 0$ or there is some stable set $S \subseteq V$ with $Q(G, S, \vec p) < 0$. Then there is some probability space $\Omega$ and events $B_1, \dots, B_n \subseteq \Omega$ such that $\Pr_{\Omega}(B_i) = p_i$ for $i = 1, \dots, n$ and such that $G$ is a dependency graph for $\mathcal B = \{B_1, \dots, B_n \}$ and $\Pr(\overline{\mathcal B}) = 0$.

In this case, we say that \emph{Shearer's criterion is violated} by $G, \vec p$.
\end{enumerate}
\end{theorem}

Having bad-events with probability $0$ or $1$ is not so interesting, and Theorem~\ref{shearer} can be simplified when we disallow these cases.
\begin{theorem}[\cite{harvey}, Lemma 5.27]
\label{shearer2}
Suppose that $p_1, \dots, p_n \in (0,1)$. Shearer's criterion is satisfied by $G, p$ if and only if $Q(G, S, \vec p) > 0$ for all stable sets $S$.
\end{theorem}

Thus, Shearer's criterion exactly characterizes which probability and lopsidependency structure of the bad-events  guarantees a positive probability of avoiding $\mathcal B$.  From a theoretical point of view, alternate bounds such as  Theorem~\ref{lll-basic-thm} are all weaker than, and are implied by, Shearer's criterion. However,  Shearer's criterion is technically difficult to apply to constructions in combinatorics. 

\subsection{The variable-assignment LLLL}
The LLLL has been applied to diverse probability spaces such as random permutations \cite{lu-szekeley}, Hamiltonian cycles \cite{albert}, and perfect matchings \cite{lu-szekeley2}. However, by far the most common form of the LLL and LLLL concerns what we refer to as the \emph{variable-assignment} setting.  Here, the probability space $\Omega$ has $m$ independent discrete random variables $X_1, \dots, X_m$, and the bad-events can be taken to be ``monomial events''; that is, each $B \in \mathcal B$ can be written in the form
$$
(X_{i_1} = j_1) \wedge (X_{i_2} = j_2) \wedge \dots \wedge (X_{i_k} = j_k)
$$

For such a monomial event, we define $\text{var}(B) = \{i_1, \dots, i_k \}$.  We say that two events $B, B'$ \emph{disagree} on variable $i$ if $B$ demands $X_i = j$ and $B'$ demands $X_i = j'$ for $j \neq j'$. 

\begin{Definition}
The \emph{canonical dependency graph} $G$ has an edge $(B, B')$ iff $\text{var}(B) \cap \text{var}(B') \neq \emptyset$. The \emph{canonical lopsidependency graph} $G$ has edge $(B, B')$ iff $B$ disagrees with $B'$ on any variable $i \in \text{var}(B) \cap \text{var}(B')$.
\end{Definition}

It is immediate that the canonical dependency graph is, indeed, a dependency graph for $\Omega, \mathcal B$. The fact that the canonical lopsidependency graph is a lopsidependency graph follows from the FKG inequality. Most applications of the LLL use only the canonical dependency graph; some noteworthy applications of the canonical lopsidependency graph include monochromatic hypergraph coloring \cite{mcdiarmid} and boolean satisfiability \cite{gst}.  We will discuss the latter in much more detail later.

In \cite{kolipaka}, Kolipaka \& Szegedy noted that the Shearer criterion is not tight for the variable-assignment LLL setting. Namely, they found an explicit dependency graph and vector of probabilities where the Shearer criterion is violated yet any variable-assignment realization must have a satisfying assignment. Later work \cite{he3} provided a more systematic description of which dependency graphs were satisfiable in the variable-assignment setting.

\subsection{The Moser-Tardos algorithm}
The LLLL ensures that $\Pr(\overline {\mathcal B}) > 0$, and this is usually sufficient for combinatorics where the main goal is to show existence results. However, typically $\Pr( \overline {\mathcal B})$ is exponentially small, and hence the LLLL does not give efficient algorithms for \emph{constructing} such a configuration. In \cite{moser-tardos}, Moser \& Tardos introduced a remarkably simple algorithm  for the variable-assignment LLLL setting:

\begin{algorithm}[H]
\centering
\begin{algorithmic}[1]
\State Draw each variable independently from the distribution $\Omega$.
\While{there is a true bad-event on $X$}
\State  Choose a true bad-event $B$ arbitrarily.
\State Resample $\text{var}(B)$ according to the distribution $\Omega$.
\EndWhile
\end{algorithmic}
\caption{The Moser-Tardos (MT) algorithm}
\end{algorithm}

They showed that when the asymmetric LLLL criterion is satisfied with respect to the canonical lopsidependency graph, then this algorithm terminates in expected polynomial time with a configuration avoiding $\mathcal B$. Later work \cite{kolipaka} showed that this algorithm terminates quickly whenever the Shearer criterion is satisfied. Thus, at least for the variable-assignment LLLL setting, this gives an efficient algorithm for nearly every construction based on the LLLL.

In \cite{harris2}, Harris gave a different type of criterion for the Moser-Tardos algorithm. Unlike the symmetric LLLL or other similar criteria, this cannot be stated solely in terms of the dependency graph and the probabilities of the bad-events.  We summarize it here (in a slightly simplified form).

\begin{Definition}[Orderability] 
\label{ord-def} Given $B \in \mathcal B$, we say that a set of bad-events $Y \subseteq \mathcal B$ is \emph{orderable} to $B$, if  there is an ordering $Y = \{B_1, \dots, B_s \}$, such that, for each $i = 1, \dots, s$, there is a variable $z_i \in \text{var}(B) \cap \text{var}(B_i)$ where $B$ disagrees with $B_i$ on $z_i$ but $B$ does not disagree with $B_1, \dots, B_{i-1}$ on $z_i$.
\end{Definition}

\begin{theorem}[\cite{harris2}]
\label{var-assignment-thm}
Suppose there is $\mu: \mathcal B \rightarrow [0, \infty)$ satisfying the condition
$$
\forall B \in \mathcal B \qquad  \mu(B) \geq \Pr(B) \Bigl( \mu(B) + \sum_{\substack{\text{$Y$ orderable}\\\text{to $B$}}}  \prod_{A \in Y} \mu(A) \Bigr)
$$

Then the Moser-Tardos algorithm terminates with probability 1.
\end{theorem}

Theorem~\ref{var-assignment-thm} is superficially similar to the cluster-expansion criterion. It is strictly stronger than the asymmeric LLLL and certain simplified forms of the cluster-expansion criterion. However, its relation to the Shearer criterion is not clear.  It is quite plausible, along the lines of \cite{kolipaka, he}, that it truly takes advantage of extra information in the variable assignment LLLL. On the other hand it is quite plausible that Theorem~\ref{var-assignment-thm} is more along the lines of \cite{kolipaka2}, namely, it provides a more accurate and computationally efficient approximation to Shearer's criterion.

In this paper, we will construct a problem instance for which Theorem~\ref{var-assignment-thm} is satisfied, yet Shearer's criterion is violated. Thus, it is impossible to deduce the fact that $\Pr(\overline {\mathcal B}) > 0$ based only on the probabilities and interdependency structure of the bad-events; it is necessary to take into account the decomposition of the bad-events into variables (as is provided by Theorem~\ref{var-assignment-thm}). In other words, Theorem~\ref{var-assignment-thm} can be stronger than Shearer's criterion.

We emphasize that Shearer's criterion concerns arbitrary probability spaces; one cannot hope to provide a stronger criterion than Shearer's \emph{for the level of generality to which the latter applies}. The strength of Theorem~\ref{var-assignment-thm} comes from its \emph{less general} setting (the variable assignment LLLL), which is nevertheless encompasses many applications in combinatorics.

We also remark on other related criteria for the variable-assignment LLL setting. For instance, \cite{he,he2} derive certain convergence conditions in terms of the bipartite  graph $H$ on vertex sets $\{1, \dots, m \}$ and  $\mathcal B$ and an edge on $(i, B)$ when $i \in \text{var}(B)$, and \cite{he3} derives conditions in terms of the probabilities that certain neighboring bad-events hold simultaneously.

\section{Satisfiability with bounded variable occurrence}
Consider boolean $k$-satisfiability instances, where we have $m$ boolean variables $X_1, \dots, X_m$ and $n$ clauses $C_1, \dots, C_n$ of width $k$, each of the form
$$
C_i \equiv l_{i1} \vee l_{i2} \vee \dots \vee l_{ik}
$$
for distinct literals $l_{i1}, \dots, l_{ik}$  (i.e. expressions of the form $X_j$ or $\neg X_j$). The goal is to produce a value for the boolean variables $X_1, \dots, X_m  \in \{T, F \}^m$ such that all the clauses $C_i$ are simultaneously true. Equivalently, we want to find a satisfying assignment of the conjuctive-normal form (CNF) formula 
$$
\Phi = \bigwedge_{i=1}^n l_{i1} \vee l_{i2} \vee \dots \vee l_{ik}
$$

We are interested specifically in instances where each variable appears in a bounded number of clauses.  For each $i = 1, \dots, m$, define $R_0(\Phi,i)$ and $R_1(\Phi,i)$ to be the number of clauses which contain the literal $X_i$ (respectively $\neg X_i$), and let $R(\Phi,i) = R_0(\Phi,i) + R_1(\Phi,i)$.   In \cite{kratochvil}, Kratochv\'{i}l, Savick\'{y}, and Tuza defined the function $f(k)$ as the largest integer $L$ such that whenever $R(\Phi,i) \leq L$ for all $i$, then $\Phi$ is satisfiable; they showed $f(k) \geq \frac{2^k}{e k}$. A series of later works \cite{savicky-sgall, hoory, gebauer, gst} showed a variety of upper and lower bounds of $f(k)$. In particular, \cite{gst} showed
$$
\Big \lfloor \frac{2^{k+1}}{e (k+1)} \Big \rfloor \leq f(k) \leq (1 + O(k^{-1/2})) \frac{2^{k+1}}{e k},
$$

The lower bound comes from the variable-assignment LLLL. Here, the probability space $\Omega$ is defined by setting each variable $X_i = T$ with a certain probability $p_i$ given by
$$
p_i = 1/2 + x \frac{ R_1(\Phi,i) - R_0(\Phi,i) }{R(\Phi,i)}
$$
for some carefully chosen parameter $x \geq 0$. Then, for each clause $C_i$, there is a corresponding bad-event $B_i$ that $C_i$ is false, namely $B_i$ has the form
$$
(X_{i1} = j_{i1}) \wedge \dots \wedge (X_{ik} = j_{ik})
$$
where $j_{i1}, \dots, j_{ik} \in \{T, F \}$.  Using Theorem~\ref{var-assignment-thm} in place of the LLLL, and using a slightly different value for the probabilities $p_i$, Harris \cite{harris2} showed a stronger bound
\begin{equation}
\label{harris-fk-eqn}
f(k) \geq \frac{2^{k+1} (1 - 1/k)^k}{k-1} - \frac{2}{k}
\end{equation}

With these constructions, we thus know the asymptotic bound
$$
f(k) \sim \frac{2^{k+1}}{e k};
$$
nevertheless, there are two main reasons to determine $f(k)$ as precisely as possible. First, since $f(k)$ grows exponentially in $k$,  the asymptotic value is not as relevant for practical applications. Second, \cite{kratochvil} showed a sudden gap in the computational complexity of $k$-SAT: for problem instances where variables may appear in $f(k)+1$ clauses, it is NP-complete to determine satisfiability. On the other hand, problems instances where they appear in at  most $f(k)$ clauses are always satisfiable and the problem is computationally vacuous. Thus, tiny gaps in the value of $f(k)$ can lead to huge gaps in computational hardness.

\subsection{Restricting the number of occurrences of each literal}
Our goal  is to demonstrate that the bound in Eq.~(\ref{harris-fk-eqn}) cannot be shown from the Shearer criterion. If the probability space $\Omega$ is allowed to vary in a problem-specific way, then any satisfiable formula can trivially satisfy the LLL: namely, $\Omega$ puts probability mass $1$ on some satisfying assignment. Thus, in order to separate the LLL and Theorem~\ref{var-assignment-thm}, we must restrict $\Omega$ to be problem-independent.

In both the constructions of \cite{gst} and \cite{harris2}, the probabilities $p_i$ depend solely on the imbalance between $R_0(\Phi,i)$ and $R_1(\Phi,i)$. They use slightly different formulas; however, in both constructions, the extremal case is when $R_0(\Phi,i) = R_1(\Phi,i)$, in which case $p_i$ is set to $1/2$. 

Accordingly,  let us define $f'(k)$ to be the largest integer $L$ such that whenever $R_0(\Phi,i) \leq L$ and $R_1(\Phi,i) \leq L$  for all $i$, then the formula $\Phi$ is satisfiable. Clearly $f'(k) \geq f(k)/2$. This function is also studied in \cite{gebauer}, with slightly different terminology, in terms of a combinatorial object called a $(k,d)$-tree. 
\begin{Definition}[\cite{gst}]\negthickspace \negthickspace \footnote{The definitions of $(k,d)$-trees are slightly shifted in the two papers; the object referred to as a $(k,d)$-tree in \cite{gst} is referred to as a $(k-1,d)$-tree in \cite{gebauer}. To put things on a consistent footing, we have adopted the terminology of \cite{gst}.}
A $(k,d)$-tree is a binary tree $T$ where every leaf has depth at least $k$, and every node $u$ of $T$ has at most $d$ descendant leaves within distance $k$ of $u$.
\end{Definition}
We quote the following two results from \cite{gebauer} and \cite{gst}:
\begin{theorem}
\begin{itemize}
\item \cite[Lemma 2]{gebauer} If there exists a $(k-1,d)$-tree, then there is an unsatisfiable $k$-CNF formula where every literal occurs in at most $d$ clauses.
\item \cite[Theorem 1.3]{gst} For any $k \geq 1$, there exists a $(k,d)$ tree with $d = (2/e + O(k^{-1/2})) 2^k/k$
\end{itemize}
\end{theorem}

This immediately gives the following result:
\begin{theorem}
$f'(k) \leq (1 + O(k^{-1/2})) \frac{2^k}{e k}$
\end{theorem}

Let us use the LLL and Theorem~\ref{var-assignment-thm} to show more precise lower bounds on $f'(k)$.  We will fix a problem-independent probability space $\Omega$ to set each $X_i$ to be T with probability $p_i = 1/2$. For each clause $C_i$, we have a bad-event $B_i$ with probability $\Pr(B_i) = p = 2^{-k}$.
\begin{theorem}[Follows easily from the symmetric LLLL]
\label{fl-bound}
$f'(k) \geq \lfloor \frac{2^k}{e k} - 1/k \rfloor$
\end{theorem}
\begin{proof}
Consider some bad-event, without loss of generality
$$
B \equiv (X_1 = T) \wedge \dots \wedge (X_k = T)
$$

The neighbors of $B$ in the canonical lopsidependency graph $G$ are bad-events involving $X_i = F$ for some $i = 1, \dots, k$; as each literal occurs at most $L$ times, there are at most $d = k L$ such bad-events.  The symmetric LLLL criterion $e p (d+1) \leq 1$ then holds if $L \leq \frac{2^k}{e k} - 1/k$.
\end{proof}

\begin{theorem}[From Theorem~\ref{var-assignment-thm}]
\label{fm-bound}
Suppose that 
$$
R_0(\Phi,i), R_1(\Phi,i) \leq \frac{ (2^k-1)(1-1/k)^{k-1}}{k}
$$
for all $i$. Then  the Moser-Tardos algorithm finds a satisfying assignment of $\Phi$ in expected polynomial time. In particular,
$$
f'(k) \geq \Big \lfloor \frac{ (2^k-1)(1-1/k)^{k-1}}{k} \Big \rfloor
$$
\end{theorem}
\begin{proof}
We will set $\mu(B) = \alpha$ for all $B \in \mathcal B$, where $\alpha \geq 0$ is some parameter to be determined. Consider some bad-event, without loss of generality
$$
B \equiv (X_1 = T) \wedge \dots \wedge (X_k = T)
$$

It is difficult to list all orderable sets of neighbors of $B$ according to Definition~\ref{ord-def}. However, to apply Theorem~\ref{var-assignment-thm}, we only need to provide an \emph{upper bound} on the sum over such orderable sets (possibly including some additional neighbor-sets $Y$ as well). Any such orderable set will have, for each $j = 1, \dots, k$, a choice of zero or one bad-events $A_j$ which first disagree with $B$ on variable $X_j$. (That is, in Definition~\ref{ord-def}, we have $B_i = A_j$ where $z_i = X_j$). Thus, we have an upper bound:
$$
\sum_{\substack{\text{$Y$ orderable}\\\text{to $B$}}} \prod_{A \in Y} \mu(A) \leq  \prod_{j=1}^k (1 + R_1(\Phi,j) \alpha) \leq (1 + L \alpha)^k
$$

So a  sufficient criterion to satisfy Theorem~\ref{var-assignment-thm} is
\begin{equation}
\label{ff2}
\alpha \geq 2^{-k} (\alpha + (1 + L \alpha)^k)
\end{equation}

We choose $\alpha$ to maximize $\alpha - 2^{-k} (\alpha + (1 + L \alpha)^k)$; simple calculus gives $\alpha = \frac{\bigl(\frac{2^k-1}{k L}\bigr)^{\frac{1}{k-1}}-1}{L}$, which is non-negative for $L \leq \frac{2^k - 1}{k}$. With this choice of $\alpha$, the condition (\ref{ff2}) is satisfied for 
$$
L \leq \frac{ (2^k-1)(1-1/k)^{k-1}}{k}
$$

Thus, if $L \leq  \frac{ (2^k-1)(1-1/k)^{k-1}}{k}$ and $L \leq \frac{2^k-1}{k}$, then Theorem~\ref{var-assignment-thm} is satisfied. The second condition $L \leq \frac{2^k - 1}{k}$ can be easily seen to be redundant, leading to the given bounds.
\end{proof}

In either case, we have $f'(k) \sim 2^k/(e k) \sim f(k)/2$. Let us define $F_{\text{LLL}}(k) = \lfloor \frac{2^{k}}{e k} - 1/k \rfloor$ and $F_{\text{MT}}(k) = \lfloor \frac{ (2^k-1) (1 - 1/k)^{k-1}}{k} \rfloor$ to be the bounds on $f'(k)$ which are provable respectively from the symmetric LLLL (Theorem~\ref{fl-bound}) and from the criterion of Theorem~\ref{fm-bound}. We observe that 
$$
F_{\text{MT}}(k) - F_{\text{LLL}}(k) \geq \frac{2^k}{2 e k^2} - 1
$$

So the gap between the LLL and Theorem~\ref{var-assignment-thm} appears to be growing exponentially in $k$. (The relative difference between the formulas approaches zero, however).
 
\section{Constructing the extremal formula $\Phi$}
Let us fix integers $L,k$.  We will  construct a $k$-SAT instance $\Phi$ with $R_0(\Phi,i), R_1(\Phi,i) \leq L$, in which the Shearer criterion is \emph{violated} for the canonical lopsidependency graph corresponding to the natural space $\Omega$ where $\Pr(X_i = T) = 1/2$, and all variables $X_i$ are  independent, and with the natural collection of bad-events corresponding to the clauses.  However, $L \leq F_{\text{MT}} (k)$; thus Theorem~\ref{var-assignment-thm} ensures that $\Phi$ is satisfiable.

To begin the construction, start with $\Phi_0$ containing no clauses (i.e. $\Phi_0$ is the tautology). At stage $i$ of the process, we modify $\Phi_{i-1}$ to produce a new formula $\Phi_i$ by adding $L - 1$ clauses in which $i$ appears positively and $L - 1$ clauses in which $i$ appears negatively. All other variables in these clauses are completely new, not appearing in any clause of $\Phi_{i-1}$; they all appear positively in the $2L-2$ new clauses, and each of the new variables (other than variable $i$) appears in exactly one new clause. 

Note that $\Pr(B) = p = 2^{-k}$ for all bad-events.  Furthermore, since each variable $i$ has exactly one positive occurence added in some iteration $\Phi_{i'}$ for $i' \neq i$, we have
$$
R_0(\Phi_j, i) \leq L \qquad R_1(\Phi_j, i) \leq L - 1
$$
for all $i,j$.

Define $G_{\ell}$ be the canonical lopsidependency graph corresponding to the bad-events for the formula $\Phi_{\ell}$. Although these graphs are complex, they contain a relatively simple and regular  family of subgraphs $H_j$. We will show that Shearer's criterion is violated for these subgraphs; as shown in \cite{shearer}, this implies that Shearer's criterion is violated for the overall graph $G_{\ell}$.

The graph family $H_j$ will consist of many copies of $K_{L-1, L-1}$, the complete bipartite graph with $L - 1$ vertices on each side. Each graph $H_j$ has a special copy of $K_{L-1, L - 1}$, called the \emph{root} of $H_j$.  We define these graphs recursively. First, $H_0$ is the empty graph. To form $H_{j+1}$, we start by taking a new copy of $K_{L-1, L - 1}$ designated as the root of $H_{j+1}$.  For each vertex $v$ in this root, we add $k-1$ separate new copies of $H_j$, along with an edge connecting $v$ to all the vertices in the right-half of the root of the corresponding $H_j$.

For example, $H_1$ consists of a single copy of $K_{L - 1, L - 1}$.  See Figure~\ref{fig1}.
\begin{figure}[H]
\vspace{1.0in}
\begin{center}
\includegraphics[trim = 0.5cm 21.5cm 6.5cm 5cm,scale=0.5,angle = 0]{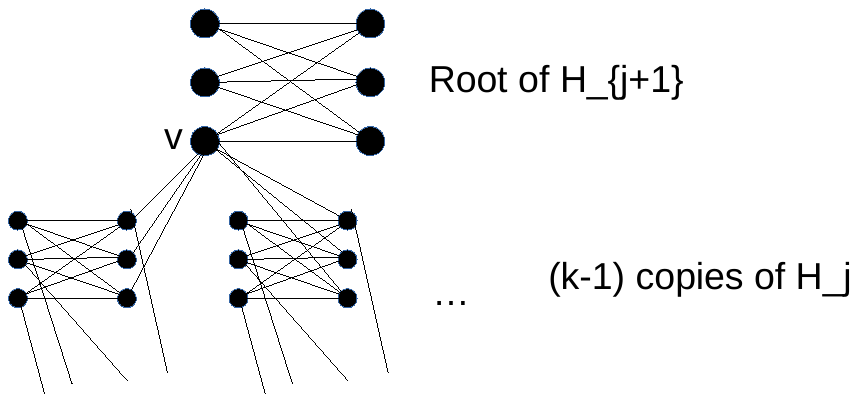}
\vspace{0.2in}
\caption{Construction of $H_{j+1}$ from $H_j$. We have only shown here two copies of $H_j$ corresponding to a single vertex $v$ in the root of $H_{j+1}$. There are $k-1$ copies of $H_j$ for each vertex in the root of $H_{j+1}$ (a total of $2 (L - 1) (k-1)$ copies of $H_j$).
\label{fig1}}
\end{center}
\end{figure}

\vspace{-0.3in}

\begin{proposition}
\label{hjprop}
Any graph $H_j$ appears as a subgraph of $G_{\ell}$ for some $\ell$ sufficiently large.
\end{proposition}
\begin{proof}
Define $A_i$ to be the collection of clauses in $\Phi_i$ but not $\Phi_{i-1}$. We can also define a tree structure $\mathcal T$ on the variables of $\Phi$: variable $i$ is a parent of variable $j$ if variable $j$ appears in $\Phi_i$ but not $\Phi_{i-1}$. For any variable $i$, let $\mathcal T_i$ denote the subtree of $\mathcal T$ rooted at $i$.

For any set of variables $S$, define $G_{\ell}[S]$ to be the subgraph of $G_{\ell}$ induced on the clauses $\phi$ of $\Phi_{\ell}$ such that all variables in $\phi$ come from $S$. Observe that if $S, S'$ are disjoint sets of variables then $G_{\ell}[S], G_{\ell}[S']$ are also vertex-disjoint graphs.

We will prove by induction on $j$ a stronger claim: for any variable $i$, there is some integer $D(i,j)$ sufficiently large such that the induced subgraph $G_{D(i,j)}[ \mathcal T_i ]$ contains a copy of $H_j$, and the root of this copy of $H_j$ corresponds to the clauses of $A_i$.

When $j = 0$ this is vacuously true. For the induction step, consider some variable $i$. Let $C$ denote the $(2 L - 2) (k-1)$ variables which are children of $i$ in $\mathcal T$.  By inductive hypothesis, for each $i' \in C$, the graph $G_{D(i',j-1)}[ \mathcal T_{i'} ]$ contains a copy of $H_{j-1}$ whose root corresponds to $A_{i'}$. 

Let $\ell= i + \max_{i' \in C} D(i', j-1)$; we claim that the choice $D(i,j) = \ell$ satisfies the induction claim. For, in the graph $G_{\ell}[\mathcal  T_i]$, the clauses of $A_i$ in which $i$ appears positively are lopsidependent with those clauses in which $i$ appears negatively. Thus, it has a copy of $K_{L-1,L-1}$ corresponding to $A_i$; we denote this copy by $J$. The graph $G_{\ell}[ \mathcal T_i]$ also contains the disjoint graphs $G_{\ell}[ \mathcal T_{i'} ]$ for each $i' \in C$. For each such $i' \in C$, let $J_{i'}$ denote the corresponding copy of $H_{j-1}$ in $G_{\ell}[ \mathcal T_{i'} ]$.

Consider some clause $\phi \in A_i$, corresponding to a vertex of $J$, and some variable $i' \neq i$ in this clause. The root of $J_{i'}$ corresponds to the clauses $A_{i'}$. Note that $\phi$ is the only clause of $A_i$ in which $i'$ appears, and it appears positively in $\phi$. Variable $i'$ also appears negatively in exactly $L-1$ clauses of $A_{i'}$, which correspond to the right-half of $J_{i'}$. Thus, there are edges from $\phi$ in $J$ to all the right-vertices in $k-1$ copies of $H_{j-1}$. As this is true for every $\phi \in J$, the resulting graph is precisely $H_j$. This completes the induction.
\end{proof}

\section{Computing the Shearer criterion for $H_j$}
We now compute the Shearer criterion for the family of graphs $H_j$. For our intermediate calculations,  we  also  need to work with another closely-related family of graphs. For each $j \geq 0$, define a graph $H'_{j}$ by taking a single vertex $v$ along with $k-1$ new copies of $H_{j}$. We include an edge from $v$ to all the vertices in the right-half of the roots of $H_{j}$. See Figure~\ref{fig2}.

\begin{figure}[ht]
\vspace{0.9in}
\begin{center}
\includegraphics[trim = 0.5cm 21.5cm 6.5cm 5cm,scale=0.5,angle = 0]{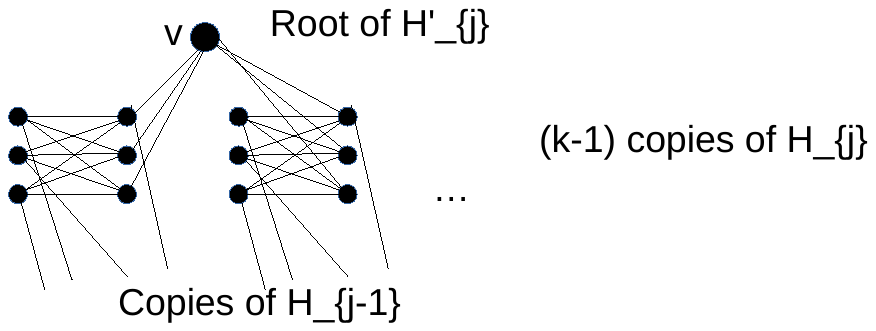}
\caption{The construction of $H'_{j}$ from $H_j$.
\label{fig2}}
\end{center}
\vspace{-0.3in}
\end{figure}

We will make use of two computational tricks for stable set polynomials; the proofs of these are elementary and are omitted here.
\begin{proposition}
\label{shear-prop1}
If vertex set $V$ is partitioned into connected-components as $V = V_1 \sqcup V_2$, then
$$
Q(G, \emptyset, \vec p) = Q(G[V_1], \emptyset, \vec p) Q(G[V_2], \emptyset, \vec p)
$$
\end{proposition}
\begin{proposition}
\label{shear-prop2}
Suppose $X \subseteq V$. Then
$$
Q(G, \emptyset, \vec p) = \sum_{\substack{\text{stable set } U \subseteq X}} Q(G[V - X - N(U)], \emptyset, \vec p) \prod_{i \in U} (-p_i) 
$$
\end{proposition}

We now begin the calculation.

\begin{proposition}
\label{rec-prop}
Let us define
$$
r_j = Q(H'_j, \emptyset, \vec p) \qquad s_j = Q(H_j, \emptyset, \vec p)
$$
Then $r_0 = 1-p, s_0 = 1$, and $r,s$ satisfy the mutual recurrence relations for $j \geq 1$:
\begin{align*}
r_j & = s_{j}^{(k-1)} -  p r_{j-1}^{(k-1) (L - 1)} s_{j-1}^{(k-1)^2 (L - 1)} \\
s_j &= 2 r_{j-1}^{(L-1)} s_{j-1}^{(k-1) (L - 1)} - s_{j-1}^{(k-1)(2 L-2)}
\end{align*}
\end{proposition}
\begin{proof}
The base cases are clear, since $H_0$ is empty and $H'_0$ is a single node.  We  first show the bound on $s_j$ for $j \geq 1$.  In any stable set $U$ of $H_j$, either $U$ contains zero vertices from the left half of the root of $H_j$, or zero vertices from the right-half of the root of $H_j$, or both. In the first two cases, when we remove the vertices in the left (respectively right) half of $H_j$, then we are left with $L - 1$ copies of $H'_{j-1}$ and $(k-1) (L - 1)$ copies of $H_{j-1}$. In the third case, we are left with $(k-1) (2 L - 2)$ copies of $H_{j-1}$. We can sum the first two contributions and subtract the third, as it is double-counted: this gives
$$
s_j = 2 r_{j-1}^{(L-1)} s_{j-1}^{(k-1) (L - 1)} - s_{j-1}^{(k-1)(2 L-2)}
$$

Next consider the bound for $r_j$. Let $v$ denote the root node of $H'_j$ and let $J_1, \dots, J_{k-1}$ be the copies of $H_j$ to which it is connected, and let $P_i$ denote the root of each $J_i$. We apply Proposition~\ref{shear-prop2} with $X = \{v \}$, and so either $U = \emptyset$ or $U = \{v \}$. For $U = \emptyset$, the graph $H'_j[V - X - N(U)]$ consists of $k-1$ independent copies of $H_{j}$.  For $U = \{v \}$, consider the graph $H'_j[V - X - N(U)]$: the vertices in the left half of $P_i$ now yield $L - 1$ disconnected copies of $H'_{j-1}$ and each vertex $u$ in the right half of $P_i$ now yields $k-1$ disconnected copies of $H_{j-1}$. Over all $k-1$ choices of $i$ and all $(k-1)(L-1)$ choices for $u$ in each $P_i$, we see that $H'_j[V - v - N(v)]$  consists of $(k-1) (L-1)$ copies of $H'_{j-1}$ and $(k-1)^2 (L-1)$ copies of $H_{j-1}$. See Figure~\ref{fig3}.

\begin{figure}[ht]
\vspace{0.8in}
\begin{center}
\includegraphics[trim = 0.5cm 20.5cm 6.5cm 5cm,scale=0.5,angle = 0]{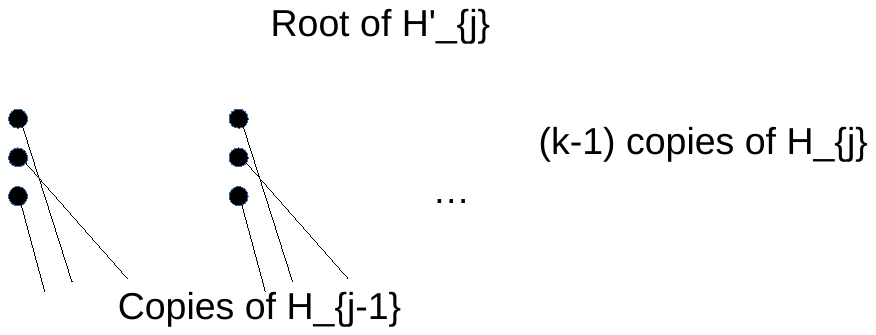}
\caption{Removing the root node from $H'_{j}$}
\label{fig3}
\end{center}
\vspace{-0.2in}
\end{figure}

Summing the contributions of these two terms according to Proposition~\ref{shear-prop2} gives
\[
r_j = Q(H'_j, \emptyset, \vec p) = s_{j}^{(k-1)}-  p r_{j-1}^{(k-1) (L - 1)} s_{j-1}^{(k-1)^2 (L - 1)}. \qedhere
\]
\end{proof}

\begin{proposition}
\label{fixed-point-prop}
Suppose that $G_{\ell}$ satisfies the Shearer condition for all $\ell \geq 0$. Then, if we define the function $g:[0,1] \rightarrow \mathbb R$ by
$$
g(a) = 1 - \frac{p}{(2 - a^{-(L - 1)})^{k-1} } ,
$$
there is some $a  \in (2^{\frac{-2}{2 L-2}},1]$ satisfying $g(a) = a$.
\end{proposition}
\begin{proof}
For $j \geq 0$ define $$
a_j = \frac{r_j}{s_j^{(k-1)}}.
$$
 We will show a recurrence relation for $a_j$. Using Proposition~\ref{rec-prop}, we calculate for $j \geq 1$:
\begin{align*}
a_j &= \frac{s_{j}^{(k-1)}-  p r_{j-1}^{(k-1) (L - 1)} s_{j-1}^{(k-1)^2 (L - 1)}}{s_{j}^{(k-1)}} = 1 - \frac{p r_{j-1}^{(k-1) (L - 1)}}{s_{j-1}^{(k-1)^2 (L - 1)}} \cdot \frac{ s_{j-1}^{(k-1)^2 (2 L-2)}}{s_{j}^{(k-1)}} = 1 - \frac{ p a_{j-1}^{(k-1)(L-1)} }{ \bigl( \frac{s_j}{ s_{j-1}^{(k-1) (2 L-2)}} \bigr)^{k-1} }
\end{align*}

Here again using Proposition~\ref{rec-prop}, we get
\begin{equation}
\label{bjeqn}
\frac{s_{j}}{s_{j-1}^{(k-1)(L-2)} }  = \frac{2 r_{j-1}^{(L-1)} s_{j-1}^{(k-1) (L - 1)} - s_{j-1}^{(k-1)(2 L-2)}}{s_{j-1}^{(k-1)(2 L - 2)}} = \frac{2 r_{j-1}^{(L-1)}}{s_{j-1}^{(k-1)(L - 1)}} - 1 = 2 a_{j-1}^{(L - 1)} - 1
 \end{equation}
 and, substituting this into the equation for $a_j$, this implies:
\begin{equation}
\label{bjeqn2}
a_j = 1 - \frac{p a_{j-1}^{(k-1) (L - 1)}}{ (2 a_{j-1}^{(L - 1)} - 1)^{k-1} } = g(a_{j-1}).
\end{equation}

We must have $a_j > 2^{\frac{-2}{2 L-2}}$ for all $j \geq 1$. For, if not, then Eq.~(\ref{bjeqn}) would otherwise imply that $\frac{s_j}{s_{j-1}^{(k-1)(2 L-2)}} \leq 0$; thus,  either $s_j \leq 0$ or $s_{j-1} \leq 0$. Thus, either $H_{j}$ or $H_{j-1}$ violates the Shearer condition, and so would some $G_{\ell}$; this contradicts our hypothesis.
 
Now suppose $g(a) < a$ for all $a \in  (2^{-\frac{-2}{2 L-2}},1]$, so from Eq.~(\ref{bjeqn2}) the sequence $a_1, a_2, \dots$ decreases monotonically. Because of the lower bound $a_j \geq 2^{\frac{-2}{2 L-2}}$, it converges to some limit point $a \geq 2^{\frac{-2}{2 L-2}}$. By continuity, this must be a fixed point, i.e. $g(a) = a$, as desired. Furthermore, since $g(a)$ diverges to infinity at $a = 2^{\frac{-2}{2 L-2}}$, we must indeed have $a > 2^{-\frac{2}{2 L-2}}$ strictly.

Otherwise, suppose that $g(a) \geq a$ for some $a \in (2^{-\frac{-2}{2 L-2}},1]$. Observe that $g(1) = 1 - p < 1$. Hence, the function $g(a) - a$ changes sign on the interval $(2^{-\frac{-2}{2 L-2}},1]$. This implies there must be a fixed point $g(a) = a$ on this interval.
\end{proof}

\begin{proposition}
\label{prop45}
Suppose 
$$
L > 1-\frac{\ln (2-t)}{\ln \left(1-2^{-k} t^{1-k}\right)}
$$
for all $t \in (2^{k/(k-1)},2)$. Then the Shearer condition is violated on $G_{\ell}$, for $\ell$ sufficiently large.
\end{proposition}
\begin{proof}
Suppose not; by Proposition~\ref{fixed-point-prop}, the function $g$ then has a fixed point $a \in (2^{-\frac{-2}{2 L-2}},1]$. So 
$$
a = 1 - \frac{2^{-k}}{(2 - a^{-(L - 1)})^{k-1} } 
$$
Solving for $L$, we thus obtain:
\begin{equation}
\label{ll-eqn}
L = 1-\frac{\ln \left(2-2^{\frac{k}{1-k}}
   (1-a)^{\frac{1}{1-k}}\right)}{\ln a}  \qquad \quad \text{for $t = 2^{k/(1-k)} (1-a)^{1/(1-k)}$}
\end{equation}
where here $t \in (2^{k/(k-1)},2)$. This contradicts our hypothesis.
\end{proof}

For any $k \geq 1$, let us define the quantity $\tilde F_{\text{Shearer}}(k)$ by:
$$
\tilde F_{\text{Shearer}}(k) = \Big \lfloor \max_{t \in (2^{k/(k-1)},2)} 1-\frac{\ln (2-t)}{\ln \left(1-2^{-k} t^{1-k}\right)} \Big \rfloor
$$

In light of Proposition~\ref{prop45}, this is an upper bound on the value of $f'(k)$ that can be shown using the LLL or any variant of it.  We observe that $\tilde F_{\text{Shearer}}(k) \geq F_{\text{LLL}}(k)$ for all values of $k$ --- this must be the case, since the bound $F_{\text{LLL}}$ was indeed derived using the LLL and this is always weaker than Shearer's criterion.  
To illustrate, we list $F_{\text{LLL}}, \tilde F_{\text{Shearer}}$, and $F_{\text{MT}}$ for a few small values of $k$.

\begin{center}
\setlength\extrarowheight{3pt}
\begin{tabular}{|c||c|c|c|}
\hline
$k$ & $F_{\text{LLL}}$ & $\tilde F_{\text{Shearer}}$ & $F_{\text{MT}}$ \\ 
\hline
9 & 20 & 21 & 22 \\
10 & 37 & 38 & 39 \\
11 & 68 & 69 & 71 \\
12 & 125 & 126 & 131 \\
13 & 231 & 233 & 241 \\
14 & 430 & 432 & 446 \\
15 & 803 & 806 & 831 \\
16 & 1506 & 1510 & 1555 \\
17 & 2836 & 2842 & 2922 \\
18 & 5357 & 5366 & 5511 \\
19 & 10151 & 10165 & 10426 \\
20 & 19287 & 19311 & 19784 \\
\hline
\end{tabular}
\end{center}
The gap between $\tilde F_{\text{Shearer}}$ and $F_{\text{LLL}}$ is very small, suggesting that there is little to no improvement possible in the bound for $f'(k)$ from a more advanced more of the LLL.  

We next derive an asymptotic approximation to $\tilde F_{\text{Shearer}}$.

\begin{proposition}
$\tilde F_{\rm{Shearer}} = \frac{2^{k}}{e k} + \Theta(\frac{2^k}{k^3})$
\end{proposition}
\begin{proof}
We can show the lower bound by taking $t = 1 - 1/k$, i.e.
$$
\tilde F_{\text{Shearer}} \geq \Big \lfloor 1-\frac{\ln (2-t)}{\ln (1-2^{-k} t^{1-k} )} \Big \rfloor \geq -\frac{\ln (2-t)}{\ln (1-2^{-k} t^{1-k})} =\frac{2^{k}}{e k} + \Omega(\frac{2^k}{k^3}).
$$

For the lower bound, let $L = \tilde F_{\text{Shearer}}(k)$, so that
$$
L \leq 1-\frac{\ln (2-t)}{\ln \left(1-2^{-k} t^{1-k}\right)}
$$
for some $t \in (2^{k/(k-1)},2)$. Using the bound $-\ln(1-x) \geq x$ for $x \geq 0$, we have:
\begin{equation}
\label{leq}
L \leq 1 + t^{k-1} 2^{k} \ln(2-t) 
\end{equation}

Since $\ln(2-t)$ is a concave-down function of $t$, we have the bound 
$$
\ln(2-t) \leq \ln(2 - t_0) + \frac{t_0 - t}{2 - t_0}
$$
for any chosen value $t_0 \in (0,2)$. Substituting this bound into (\ref{leq}), and differentiating with respect to $t$ to maximize the resulting value, we get
\begin{equation}
\label{ff3}
L \leq 1 + \frac{ \Bigl( 2 (1-1/k) (t_0 + (2 - t_0) \ln
   (2-t_0)) \Bigr)^k}{(2 - t_0) (k-1)}
\end{equation}
If we set $t_0 = 1 - 1/k$ in Eq.~(\ref{ff3}), then straightforward analysis gives:
\[
L \leq \frac{2^{k}}{e k} + O(\frac{2^k}{k^3}) \qedhere
\]
\end{proof}

On the other hand, one can easily verify that $F_{\text{MT}}(k) \geq \frac{2^k}{e k} + \Omega( \frac{2^k}{k^2})$; thus, there is a large and growing gap between $F_{\text{MT}}$ and $\tilde F_{\text{Shearer}}$.

\section{Acknowledgments}
Thanks to Aravind Srinivasan and anonymous journal referees for many helpful comments and suggestions.

\end{document}